\documentclass[12pt, twoside,a4paper]{amsart}
\usepackage[cp1251]{inputenc}
\usepackage[english,  ukrainian, russian]{babel}
\usepackage{amsmath, amsthm, amsfonts, amssymb}
\DeclareMathOperator*\uplim{\overline{lim}}
\DeclareMathOperator*\unlim{\underline{lim}}

\theoremstyle{plain}
\newtheorem{theorem}{Theorem}
\newtheorem{lemma}{Lemma}
\newtheorem{corollary}{Corollary}
\theoremstyle{definition}

\theoremstyle{remark}

\begin{document}
\pagestyle{headings}
\setcounter{page}{105}
 \hbox to \textwidth{\footnotesize\textsc Математичні Студії Volume 66, No.1 
 	\hfill
 	Matematychni Studii. V.66, No.1}
 	\noindent\textup{\scriptsize ISSN 1027-4634 (print)} \hfill \textup{\scriptsize ISSN 2411-0620 (online)}

\markboth{\textsc{A.\ A.\ Dorogovtsev, I.\ I.\ Nishchenko}}{\textsc{Visitation measures for random iterations}}
\begin{center} \textsc {A.\ A.\ Dorogovtsev, I.\ I.\ Nishchenko} \end{center}
 \begin{center} 
\bf \MakeUppercase{{Visitation measures for non-linear iterations with random perturbations}} \end{center}

{A.\ A.\ Dorogovtsev, I.\ I.\ Nishchenko}.\
 		{\it{Visitation measures for non-linear iterations with random perturbations}}, Mat.Stud.V.66, No.1(2026), 105-112\footnote{ 


\noindent{doi:10.30970/ms.66.1.105-112. This work is licensed under CC BY-NC-ND 4.0}

\noindent{*Corresponding author: I.I.Nischenko}

\hfill \copyright{A.\ A.\ Dorogovtsev, I.\ I.\ Nishchenko, 2026} %(to fill in by the Editors)
}. 
\vspace{3pt}

A sequence of non-linear iterations with non-identically distributed random perturbations is studied. The asymptotic behaviour of the sequence is described via its visitations measures, which are deterministic measures that allow one to estimate the time spent by the sequence in the neighbourhoods of its partial limits. The results obtained are compared with statements regarding the statistical convergence of the sequence of iterations under study 

\renewcommand{\refname}{References}
\renewcommand{\proofname}{Proof}

\vskip10pt

\noindent\textbf{1.~Introduction and preliminaries.} 
    
In this paper we consider nonlinear perturbed iterations of the form
\begin{equation}\label{IterationsMain}
 \begin{gathered}
 	x_{n+1} = \varphi(x_n)+\xi_{n+1},\quad n\geqslant 0, \\
     x_0=x\in\mathbb{R},
  \end{gathered}
 \end{equation}
where $\varphi\colon\mathbb{R}\to\mathbb{R}$ is a function that satisfies the Lipschitz condition on $\mathbb{R}$ 
\begin{equation}\label{LipschitzCondition}
\forall x, y \in \mathbb{R}\colon\quad |\varphi(x)- \varphi(y)|\leqslant\alpha\,|x-y|
\end{equation}
with a constant $\alpha < 1$, and $\{\xi_n,\,n\geqslant 1\}$ is a sequence of independent and non-identically distributed random variables. The aim of this work is to investigate the set of partial limits of the sequence of iterations $\{x_n,\,n\geqslant 1\}$ and to study the asymptotic properties of the sequence of counting measures 
$\{\displaystyle{\sum\limits_{k=1}^n\,\delta_{x_k},\, n\geqslant 1}\}$ that describe frequencies with which the sequence $\{x_n,\,n\geqslant 1\}$ falls within the neighbourhoods of its partial limits.

Our analysis is based on the notion of {\it a sequence of visitation measures}, introduced by A.A. Dorogovtsev in \cite{Dorogovtsev1}. 
In what follows, all the necessary definitions are provided to recall this notion. 

\noindent{\textbf{Definition 1.}}(\cite{Dorogovtsev1})
A sequence of finite measures  $\{\lambda_n,\,n\geqslant 1\}$ on $\mathbb{R}^d$ is said to be {\it increasing}  if 
	\[
	\forall A\in\mathcal{B}(\mathbb{R}^d)\; \forall n\geqslant 1\colon\quad \lambda_{n+1}(A)\geqslant\lambda_n(A),\,
	\lim\limits_{n\to\infty}\,\lambda_n(\mathbb{R}^d) = +\infty,
	\]
where $\mathcal{B}(\mathbb{R}^d)$ is the Borel $\sigma$-algebra on $\mathbb{R}^d$.

\noindent{\textbf{Definition 2.}}(\cite{Dorogovtsev1})
For an increasing sequence of measures $\{\lambda_n,\,n\geqslant 1\}$  the set  
\[
\mathfrak{M}_{\lambda}:=\biggl\{x\in\mathbb{R}^d\colon\quad \forall\varepsilon > 0\; \lim\limits_{n\to\infty}\,\lambda_n(B(x, \varepsilon)) = +\infty\biggr\}
\]
is called {\it the limit set} of the sequence  $\{\lambda_n,\,n\geqslant 1\}$. Here
$B(x, \varepsilon)$ is  an open ball of radius $\varepsilon$ centered at the point  $x$. Elements of the set $\mathfrak{M}_\lambda$ are called {\it the points of growth} of the sequence $\{\lambda_n,\,n\geqslant 1\}$.

\noindent{\textbf{Definition 3.}}(\cite{Dorogovtsev1})
Two increasing sequences of measures
$\{\lambda_n,\,n\geqslant 1\}$, $\{\mu_n,\,n\geqslant 1\}$ 	are called {\it equivalent} if the following conditions hold:
\begin{enumerate}
 \item 
$\mathfrak{M}_\lambda=\mathfrak{M}_\mu =\mathfrak{M}\neq \emptyset$;
\item
for any open sets $U_1$, $U_2$ such that  $\overline{U_1}\subset U_2$, $\mathfrak{M}\cap U_1\neq\emptyset$, $\overline{U_1}$ is a compact, the following relations hold
\begin{equation}\label{Conditions2}
\unlim\limits_{n\to\infty}\,\frac{\lambda_n(U_2)}{\mu_n(U_1)}\geqslant 1,\quad
\unlim\limits_{n\to\infty}\,\frac{\mu_n(U_2)}{\lambda_n(U_1)}\geqslant 1.
\end{equation}
\end{enumerate}

\noindent{\textbf{Definition 4.}}(\cite{Dorogovtsev1})
	Any increasing sequence of measures $\{\mu_n,\,n\geqslant 1\}$, which is equivalent to the sequence of counting measures $\{\lambda_n = \sum\limits_{k=1}^n\,\delta_{x_k},\; n\geqslant 1\}$, is called {\it a sequence of visitation  measures} of the sequence $\{x_n,\,n\geqslant 1\}$.

Let us state two theorems, proven in \cite{Dorogovtsev1}, which are important for a subsequent discussion.

\begin{theorem}(\cite{Dorogovtsev1})\label{ThmDorogovtsev1}
Let $\{\xi_n,\,n\geqslant 1\}$ be a sequence of independent random elements in $\mathbb{R}^d$ such that
\begin{equation}\label{Condition1}
\exists C\in\mathbb{R}\colon \sum\limits_{k=1}^\infty\,P\bigl(\Arrowvert\xi_k\Arrowvert > C\bigr) < +\infty.
\end{equation}
Then with probability one the non-random sequence of measures
$
\bigl\{\mu_n = \sum\limits_{k=1}^n\,P\circ\xi_k^{-1},\,n\geqslant 1\bigr\}
$
is a sequence of visitation  measures for $\{\xi_n,\,n\geqslant 1\}$.
\end{theorem}

\begin{theorem}(\cite{Dorogovtsev1})\label{ThmDorogovtsev2}
Let $\{\xi_n,\,n\geqslant 1\}$ be a sequence of independent random elements in $\mathbb{R}^d$ satisfying (\ref{Condition1}) , $m\geqslant 1$ be  fixed and $F\colon \mathbb{R}^{d(m+1)}\to\mathbb{R}^l$ be a Borel function that maps bounded sets to bounded sets. Then for  the sequence of random elements in $\mathbb{R}^l$
\[
\eta_n = F(\xi_n, \xi_{n+1}, \dots, \xi_{n+m}),\quad n\geqslant 1
\]
the non-random sequence of measures
$
\bigl\{\mu_n = \sum\limits_{k=1}^n\,P\circ\eta_k^{-1},\, n\geqslant 1\bigr\}
$
is with probability one  a sequence of its visitation  measures.
\end{theorem}

From the theorems given above, it is clear that visitation measures of a sequence of random variables are much simpler than its counting measures, yet contain information about the time spent by the sequence  in the neighbourhoods of its partial limits. In  \cite{Dorogovtsev1} this notion is applied to the study of the asymptotic behaviour of linear perturbed iterations
\[
\begin{gathered}
	x_{n+1} = Ax_n +\xi_{n+1},\quad n\geqslant 0,\\
	x_0=u\in\mathbb{R}^d
\end{gathered}
\]
with a linear operator $A$ in $\mathbb{R}^d$ and non-identically distributed perturbations $\{\xi_n,\,n\geqslant 1\}$.  In \cite{Dorogovtsev2} a similar technique is used to prove a local ergodic theorem for a sequence of nonlinear perturbed iterations of the form
\[
\begin{gathered}
	x_{n+1} = \varphi(x_n) +\xi_{n+1},\quad n\geqslant 0,\\
	x_0=u\in\mathbb{R}^d,
\end{gathered}
\]
with continuous function $\varphi\colon\mathbb{R}^d\to\mathbb{R}$. This theorem establishes the conditions under which for any continuous function $f\colon\mathbb{R}^d\to\mathbb{R}$ with compact support there exists a finite limit
\[\lim\limits_{n\to\infty}\,\frac{1}{\gamma_n}\,\sum\limits_{n\to\infty}\,f(x_k),
\]
where the normalizing sequence $\{\gamma_n,\,n\geqslant 1\}$ is obtained from the sequence of visitation measures of $\{x_n,\,n\geqslant 1\}$ and may grow more slowly than $n$.

In \cite{Vlasenko} it is demonstrated that  investigating visitation measures for $\{x_n/a_n,\,n\geqslant 1\}$ with various numerical sequences $\{a_n,\,n\geqslant 1\}$ that monotonically decrease to zero, turns out to be an effective method to study the nature of convergence to zero of the sequence of random variables $\{x_n,\,n\geqslant 1\}$ itself. 

Note that visitation measures entirely determine the frequencies with which a sequence of random variables visits the neighbourhoods of its partial limits. The idea of counting the number of times a sequence falls within the neighbourhoods of a given point goes back to  H.Steinhaus \cite{Steinhaus} and H. Fast \cite{Fast}, who independently introduced the concept of statistical convergence. This type of convergence has since been extensively developed and investigated by many authors.  We recall that a sequence $\{a_n,\,n\geqslant 1\}$ of elements of a metric space $(\mathcal{X}, d)$  \emph{converges  statistically} to an element $a\in\mathcal{X}$, if for any $\varepsilon > 0$
\begin{equation}
    \frac{1}{n}\,\sum\limits_{k=1}^n\,\mathbb{I}_{\{d(a_k, a) > \varepsilon\}} \, \to 0,\quad n\to\infty.
\end{equation}
This type of convergence allows one to identify  <<the main>> among the possible partial limits of a sequence, whose neighbourhoods are visited much more frequently than the neighbourhoods of the others. Note that the notion  of visitation measures is far more general and  informative, allowing  one to say how often a sequence of random variables visits the neighbourhoods of any of its partial limits. In particular, a conclusion can be drawn about the existence of a statistical limit.

\noindent{\textbf{Example.}}
Let $\{\eta_n,\,n\geqslant 1\}$ be a sequence of {\it iid} standard Gaussian random variables, and let
$\xi_n = \eta_n/\sqrt{2\ln (n+1)},\, n\geqslant 1
$.
Using the Borel-Cantelli lemma, it is straightforward to show that  with probability one the segment  $[-1, 1]$ is the set of partial limits of the sequence $\{\xi_n,\,n\geqslant 1\}$.  According to Theorem \ref{ThmDorogovtsev1} the sequence  $\{\sum\limits_{k=1}^n\,P\circ \xi_k^{-1},\,n\geqslant 1\}$
is a sequence of visitation measures for the sequence $\{\xi_n,\,n\geqslant 1\}$. Therefore, for 
any $\varepsilon\in (0, 1)$
\[
\lim\limits_{n\to\infty}\,\frac{1}{n}\,\sum\limits_{k=1}^n\,\mathbb{I}_{\{|\xi_k| > \varepsilon\}} = \lim\limits_{n\to\infty}\,\frac{1}{n}\,
\sum\limits_{k=1}^n\,P(|\xi_k| > \varepsilon) = 0
\]
and thus one immediately  deduces the fact of statistical convergence of $\{\xi_n,\,n\geqslant 1\}$ to zero,
\[
S-\lim\limits_{n\to\infty}\,\xi_n = 0,\quad a.s.
\]
The use of visitation measures allows us not only to find zero as <<the essential>>  partial limit of $\{\xi_n,\,n\geqslant 1\}$, but also to determine the asymptotic behaviour of the frequencies with which this sequence falls within the neighbourhoods of other partial limits. Namely, by applying to the probability
\[
P(\xi_n \in [\alpha, \beta]) = P\left(\eta_n \in \bigl[\sqrt{2\ln (n+1)}\cdot\alpha, \sqrt{2\ln (n+1)}\cdot\beta\bigr]\right)
\]
the well-known estimate 
\[
\left(\frac{1}{x}-\frac{1}{x^3}\right)\,\cdot \frac{1}{\sqrt{2\pi}}\,e^{-x^2/2}\leqslant P(\eta > x)\leqslant \frac{1}{x}\cdot \frac{1}{\sqrt{2\pi}}\,e^{-x^2/2},\quad x >0
\]
for a standard Gaussian random variable $\eta$, one can verify that for any $0 <  \alpha <\beta < 1$ and any $\varepsilon > 0$
\[
\begin{gathered}
\frac{1}{n^{1-\alpha^2+\varepsilon}}\,
\sum\limits_{k=1}^n\,P\bigl(\xi_k \in (\alpha, \beta)\bigr)\,\to  0, \quad n\to\infty, \\
\frac{1}{n^{1-\alpha^2 -\varepsilon}}\,
\sum\limits_{k=1}^n\,P\bigl(\xi_k \in (\alpha, \beta)\bigr)
\,\to  \infty, \quad n\to\infty,
\end{gathered}
\]
and thus due to the equivalence of counting and corresponding visitations measures
\[
\begin{gathered}
\frac{1}{n^{1-\alpha^2+\varepsilon}}\,
\sum\limits_{k=1}^n\,\mathbb{I}_{\{\xi_k \in (\alpha, \beta)\}}\,\to  0, \quad n\to\infty,\quad a.s.\\
\frac{1}{n^{1-\alpha^2 -\varepsilon}}\,
\sum\limits_{k=1}^n\,\mathbb{I}_{\{\xi_k \in (\alpha, \beta)\}}
\,\to  \infty, \quad n\to\infty,\quad a.s.
\end{gathered}
\]
Moreover, for any $a<\alpha_1<\beta_1<\alpha_2<\beta_2<1$
\[
\lim\limits_{n\to\infty}\,\frac{\sum\limits_{k=1}^n\,P\bigl(\xi_k \in (\alpha_2, \beta_2)\bigr)}{\sum\limits_{k=1}^n\,P\bigl(\xi_k \in (\alpha_1, \beta_1)\bigr)} =0,
\]
and therefore the same asymptotic behaviour holds for the counting measures
\[
\lim\limits_{n\to\infty}\,\frac{\sum\limits_{k=1}^n\,\mathbb{I}\bigl(\xi_k \in (\alpha_2, \beta_2)\bigr)}{\sum\limits_{k=1}^n\,\mathbb{I}\bigl(\xi_k \in (\alpha_1, \beta_1)\bigr)} = 0,\quad a.s.
\]

\noindent{\textbf{2.~Main results.}}
For any $m\geqslant 1$ define the   sequence 
$\{x_n^m,\,n\geqslant m\}$ as follows
\begin{equation}\label{TruncatedSequence1}
	x^m_n=\xi_{n}+\varphi(\xi_{n-1}+\varphi(\xi_{n-2}+\dots+\varphi(\xi_{n-m+1}+\varphi(x_0))\dots),\quad n\geqslant m.
\end{equation}	

\begin{lemma}\label{LemmapropertiesTruncatedSeq}
Let $\{\xi_n,\,n\geqslant 1\}$ be a sequence of independent random variables in $\mathbb{R}$ satisfying (\ref{Condition1}). Then
\begin{equation}    \lim\limits_{m\to\infty}\,\sup\limits_{n\geqslant m}\,|x^m_n - x_n| =0 \quad a.s.
\end{equation}
\end{lemma}
\begin{proof}
Note that  for any $n\geqslant m$
\[
|x^m_n-x_n|\leqslant \alpha\,|x^{m-1}_{n-1}-x_{n-1}|\leqslant\dots\leqslant \alpha^m\,|x_0 - x_{n-m}|\leqslant
\alpha^m\,(|x_0|+\sup_{n\geqslant 1}\,|x_n|).
\]
From condition (\ref{Condition1}) it follows that
\[
\sup_{n\geqslant 1}\,|\xi_n| \,< +\infty \quad a.s.
\]
and then, since $\alpha<1$,
\[
\sup_{n\geqslant 1}\,|x_n|\leqslant \frac{1}{1-\alpha} \,\sup_{n\geqslant 1}\,|\xi_n|< +\infty \quad a.s.
\]
Hence,
\[
\lim\limits_{m\to\infty}\,\sup\limits_{n\geqslant m}\,|x^m_n - x_n| =0 \quad a.s.
\]
\end{proof}

\begin{lemma}\label{VisitationMeasuresTruncated}
Let $\{\xi_n,\,n\geqslant 1\}$ be a sequence of independent random variables in $\mathbb{R}$ satisfying (\ref{Condition1}). Then
for any $m\geqslant 1$ the sequence 
$\{x_n^m,\,n\geqslant m\}$  with probability one has as its visitation measures the sequence
$\{\varkappa_n^m=\sum\limits_{k=m}^n\,P\circ (x^m_k)^{-1},\,n\geqslant m\}$.
\end{lemma}

\begin{proof}
Note that $\{x_n^m,\,n\geqslant m\}$ can  be represented as
\begin{equation}\label{TruncatedSequence2}
x_n^m = F_m(\xi_{n-m+1}, \dots,  \xi_{n}),
\end{equation}
where $F_m\colon\mathbb{R}^m\to\mathbb{R}$ for any $m\geqslant 1$ is defined by the equality
\begin{equation}\label{Function_Fm}
	F_m(u_1, \dots, u_m)=u_m+\varphi(u_{m-1}+\varphi(u_{m-2}+\dots+\varphi(u_1+\varphi(x_0))\dots ).
\end{equation}	
Since the function $\varphi$ in (\ref{TruncatedSequence1}) satisfies the Lipschitz condition (\ref{LipschitzCondition}), then $F_m$ is a continuous function and therefore  the statement of the lemma follows from Theorem \ref{ThmDorogovtsev2}.
   
\end{proof}

 Denote by $L^x_m$ the set of partial limits of the sequence $\{x^m_n,\,n\geqslant m\}$ defined in (\ref{TruncatedSequence1}) and let $L^\xi_m$ be the set of partial limits in $\mathbb{R}^m$ of the vector sequence $\{(\xi_{n-m+1}, \dots, \xi_{n}),\,n\geqslant m\}$. Note that since the random variables $\{\xi_n,\,n\geqslant 1\}$ are independent,  according to Kolmogorov's zero-one law, the set $L^\xi_m$ is deterministic.

\begin{lemma}\label{LemmaLimitSetKsi}
 Let the sequence $\{\xi_n,\,n\geqslant 1\}$ satisfies condition (\ref{Condition1}). Then 
    \begin{equation}
    L^x_m=F_m(L^\xi_m).
    \end{equation}
\end{lemma}
\begin{proof}
    Let $u\in L^x_m$. This means that there exists a subsequence $\{n_k,\,k\geqslant 1\}$ such that 
    \[
    x^m_{n_k}\,\to u, \quad k\to\infty.
    \]
    By construction, $x^m_{n_k}=F_m(\xi_{n_k-m+1}, \dots, \xi_{n_k})$. From condition (\ref{Condition1}) it follows that the vector sequence $\{(\xi_{n_k-m+1}, \dots, \xi_{n_k}), k\geqslant 1\}$ is bounded with probability one and therefore a convergent subsequence $\{(\xi_{n_{k_l}-m+1}, \dots, \xi_{n_{k_l}}), l\geqslant 1\}$ can be chosen from it. Let 
    \[
    (\xi_{n_{k_l}-m+1}, \dots, \xi_{n_{k_l}})\,\to\,(y_1, \dots, y_m), \quad l\to\infty.
    \]
    Then $(y_1, \dots, y_m)\in L^\xi_m$ and since $F_m$ is continuous, we obtain 
    \[
    F_m(\xi_{n_{k_l}-m+1}, \dots, \xi_{n_{k_l}})\,\to F_m(y_1, \dots, y_m),\quad l\to\infty.
    \]
    Hence, $u\in F_m(L^\xi_m)$. 

     Now assume that $u\in F_m(L^\xi_m)$. Then   $u=F_m(y_1, \dots, y_m)$ for some $(y_1, \dots, y_m)\in L^\xi_m$. This means that there exists a subsequence $\{(\xi_{n_{k}-m+1}, \dots, \xi_{n_{k}}),\,k\geqslant 1\}$ converging to $(y_1, \dots, y_m)$ as $k\to\infty$. From the continuity of $F_m$ it follows that
     \[
     x^m_{n_k} = F_m(\xi_{n_k-m+1}, \dots, \xi_{n_{k}})\,\to F_m(y_1, \dots, y_m) = u.
     \]
     Thus, $u$ is a partial limit of $\{x^m_n,\,n\geqslant 1\}$, that is, $u\in L^x_m$.
     
\end{proof}
\begin{theorem}\label{LimitSet}
The set $L^x$ of partial limits of the sequence of iterations $\{x_n,\,n\geqslant 1\}$ coincides with the intersection of the closures 
\begin{equation} 
L^x = \bigcap_{k=1}^\infty\overline{\bigcup_{m = k}^\infty L^x_m}.
\end{equation}
\end{theorem}
\begin{proof}
  Let $y\in \bigcap\limits_{k=1}^\infty\overline{\bigcup\limits_{m = k}^\infty\,L^x_m}$. Then there exists a subsequence $\{y_{m_k}\in L^x_{m_k},\,k\geqslant 1\}$ such that 
  \[
  y_{m_k}\,\to\,y,\quad k\to\infty.
  \]
  For the sequence $\{x^{m_k}_n,\,n\geqslant m_k\}$ there exists a subsequence $\{x^{m_k}_{n_j},\,j\geqslant 1\}$ such that
  \[
  x^{m_k}_{n_j}\,\to\,y_{m_k},\quad j\to\infty.
  \]
  In particular, there exists a number $n_{j_k}$ for which
  \[
  |x^{m_k}_{n_{j_k}} - y_{m_k}|< \frac{1}{k}.
    \]
    Then
    \[
    x^{m_k}_{n_{j_k}} \,\to\, y,\quad k\to\infty.
    \]
 According to Lemma \ref{LemmapropertiesTruncatedSeq}
    \[
    |x^{m_k}_{n_{j_k}} - x_{n_{j_k}}|\,\to\,0, \quad k\to\infty,
    \]
    and therefore
    \[
    x_{n_{j_k}}\,\to\,y,\quad k\to\infty.
    \]
    Hence, $y$ is a partial limit of $\{x_n,\,n\geqslant 1\}$, that is, $y\in L^x$.

    Now, let  $y\in L^x$. Then there exists a subsequence $\{x_{n_k},\,k\geqslant 1\}$ such that
    \[
    x_{n_k}\,\to \,y,\quad k\to\infty,
    \]
    that is
    \[
    \forall\varepsilon > 0\;\exists k_0\; \forall k\geqslant k_0:\quad |x_{n_k}-y|<\varepsilon.
    \]
    Note that 
    \[
    \exists m_0\; \forall m\geqslant m_0\;\forall k\geqslant 1:\quad |x^m_{n_k}-x_{n_k}|<\varepsilon.
    \]
    From this it follows that
    \[
    \forall k\geqslant k_0\quad |x^{m_0}_{n_k}-y|< 2\varepsilon.
    \]
    This means that there exists a partial limit $y_{m_0}$ of $\{x^{m_0}_n,\,n\geqslant m_0\}$ such that
    \[
    |y_{m_0}-y|\leqslant 2\varepsilon.
    \]
    By choosing $\varepsilon = 1/l$, $l\geqslant 1$ we construct subsequences $\{y_{m_l},\,l\geqslant 1\}$ such that $y_{m_l}\in L^x_m$ and
    \[
    y_{m_l}\,\to\,y,\quad l\to\infty.
    \]
    Hence, for every $k\geqslant 1$
    \[ 
    y\in\overline{\bigcup\limits_{m=k}^\infty\,L^x_m}
    \]
    implying
    \[
    y\in \bigcap_{k=1}^\infty\overline{\bigcup\limits_{m=k}^\infty\,L^x_m}.
    \]
\end{proof}

\begin{lemma}\label{LemmaForTheorem}
Suppose that for a sequence $\{x_n,\,n\geqslant 0\}$ in $\mathbb{R}^d$ and for an increasing sequence of measures   $\{\mu_n,\,n\geqslant 0\}$ on $\mathcal{B}(\mathbb{R}^d)$  the following conditions hold:
\begin{enumerate}
    \item 
    the limit set $\mathfrak{M}_\mu$ is non-empty;
    \item for every positive $\varepsilon>0$ there exists a sequence 
     $\{y_n,\,n\geqslant 0\}$ in $\mathbb{R}^d$  such that 
      $\{\mu_n,\,n\geqslant 0\}$ is its sequence of visitation measures and 
     \[
     \uplim\limits_{n\to\infty}\,\Arrowvert x_n-y_n\Arrowvert \leqslant \varepsilon.
     \]
    Then  $\{\mu_n,\,n\geqslant 0\}$ is a sequence of visitation measures for  $\{x_n,\,n\geqslant 0\}$. 
\end{enumerate}
\end{lemma}
\begin{proof}
We must verify that the sequence of counting measures $\{\nu^x_n = \sum\limits_{k=0}^n\,\delta_{x_k},\,n\geqslant 0\}$ is equivalent to the increasing sequence of measures 
$\{\mu_n,\,n\geqslant 0\}$ .
 Let us first check that  the set $L^x$ of partial limits of  $\{x_n,\,n\geqslant 0\}$ coincides with $\mathfrak{M}_\mu$. Suppose that $u\in L^x$. For $\varepsilon>0$  let  $\{y_n,\,n\geqslant 0\}$ be a sequence satisfying condition 2 in the statement of the lemma. Then
 \[
 \overline{B}(u, \varepsilon)\cap L^y = \overline{B}(u, \varepsilon)\cap \mathfrak{M}_\mu \neq \emptyset.
 \]
 Hence, $u\in\mathfrak{M}_\mu$ and thus $L^x\subset \mathfrak{M}_\mu$. The reverse inclusion is checked in a similar manner. 

Now let us verify that the conditions in (\ref{Conditions2}) of Definition 3 are satisfied. To do this, we will take open sets $U_1$, $U_2$
such that $\overline{U}_1$ is compact, $\overline{U}_1\subset U_2$ and $U_1\cap\mathfrak{M}_\mu\neq \emptyset$. Then there exists $\varepsilon > 0$ such that $\overline{U}_1^{\varepsilon}\subset U_2$, where 
$\overline{U}^{\varepsilon}$ stands for the $\varepsilon$-neighbourhood of the set $\overline{U}_1$. Then for counting measures $\{\nu^x_n = \sum\limits_{k=0}^n\,\delta_{x_k},\,n\geqslant 0\}$ and $\{\nu^y_n = \sum\limits_{k=0}^n\,\delta_{y_k},\,n\geqslant 0\}$ 
we obtain the following inequalities 
\[
\unlim\limits_{n\to\infty}\,\frac{\mu_n(U_2)}{\nu^x_n(U_1)}\geqslant
\unlim\limits_{n\to\infty}\,\frac{\mu_n(U_2)}{\nu^y_n(U^\varepsilon_1)}\geqslant\unlim\limits_{n\to\infty}\,\frac{\mu_n(U_2)}{\mu_n(U^\varepsilon_1)}\geqslant 1,
\]
\[
\unlim\limits_{n\to\infty}\,\frac{\nu^x_n(U_2)}{\mu_n(U_1)}\geqslant
\unlim\limits_{n\to\infty}\,\frac{\nu^x_n(U_2)}{\nu^y_n(U^{\varepsilon/2}_1)}\geqslant\unlim\limits_{n\to\infty}\,\frac{\nu^x_n(U_2)}{\nu^x_n(U^\varepsilon_1)}\geqslant 1,
\]
that ensure the conditions of Definition 3 are satisfied.
\end{proof}
\begin{theorem}\label{ThmVisitationMeasuresIterations}
 	The sequence of measures
	$
	\bigl\{\varkappa_n = \sum\limits_{k=1}^n\,P\circ x_k^{-1},\; n\geqslant 1\bigr\}
	$
	is a sequence of visitation measures for $\{x_n,\,n\geqslant 1\}$ with probability one.
   
\end{theorem} 
\begin{proof}
For every $n\geqslant 1$ consider a set $\{\eta_{n,k}, \,1\leqslant k\leqslant n\}$ of independent random variables that are independent of the sequence $\{\xi_n,\,n\geqslant 1\}$ and such that for any $k=\overline{1, n}$
\[
\eta_{n, k}\,\overset{d}{=}\,\xi_n.
\]
For a fixed $m\geqslant 1$ let us construct a sequence $\{y^m_n,\,n\geqslant m\}$ as follows
\[
y^m_n = F_n(\eta_{1, n},\dots, \eta_{n-m, n}, \xi_{n-m+1}, \dots, \xi_n),
\]
where $F_n$ is defined in (\ref{Function_Fm}). By similar reasoning used in the proof of Theorem \ref{ThmDorogovtsev1} in \cite{Dorogovtsev1}, it can be proven that $\{\sum\limits_{k=m}^n\,P\circ (y^m_k)^{-1},\,n\geqslant m\}$ is a sequence of visitation measures for $\{y^m_n,\,n\geqslant m\}$. 

Given that $y^m_n\,\overset{d}{=}\,x_n$ for every $n\geqslant m$ and that by construction

\[
\sup\limits_{n\geqslant m}\,|x^m_n - y^m_n|\to 0,\quad m\to\infty, 
\]
\[
\sup\limits_{n\geqslant m}\,|x_n - y^m_n|\to 0,\quad m\to\infty, 
\]
the statement of the theorem follows from Lemma \ref{LemmaForTheorem}.
\end{proof}

\noindent\textbf{3.~Iterations with Gaussian perturbations.}

In this section, we apply the notion of visitation measures to the analysis of convergence of iterations
\begin{equation}\label{EqGaussIterations}
x_{n+1} = \varphi(x_n)+ \xi_{n+1}, \quad n\geqslant 0
\end{equation}
with
\[
\xi_n =\frac{\eta_{n}}{\sqrt{2\ln (n+1)}}, \quad n\geqslant 1,
\]
where $\{\eta_n,\;n\geqslant 1\}$ is a sequence of \emph{iid} standard Gaussian random variables. 
Note that if $x_*$ is the unique solution to the equation  $x=\varphi(x)$, then the sequence
\[
\widetilde{x}_n= x_n-x_*,\quad n\geqslant 1
\]
satisfies
\[
\widetilde{x}_{n+1} = \widetilde{\varphi}(\widetilde{x}_n)+ \frac{\eta_{n+1}}{\sqrt{2\ln (n+2)}}, \quad n\geqslant 0,
\]
with the function $\widetilde{\varphi}(y) = \varphi(y+x_*) - x_*$, for which $\widetilde{\varphi}(0)=0$. Therefore, without loss of generality, we will assume in what follows that $\varphi(0) =0$ in (\ref{EqGaussIterations}) and thus $x_*=0$.

\begin{lemma}\label{LemmaPconvergenceIterations}
\[
x_n\,\xrightarrow{P}\,0,\quad n\to\infty.
\]
    
\end{lemma}
\begin{proof}
Condition (\ref{LipschitzCondition}) implies
\[
\mathbb{E}\,|x_n|\leqslant\alpha^n\,|x_0|+\sum\limits_{k=1}^n\,\alpha^{n-k}\,\mathbb{E}\,|\xi_k|.
\]
From the construction of $\{\xi_k,\;k\geqslant 1\}$ and the Toeplitz theorem it follows that
\[
\sum\limits_{k=1}^n\,\alpha^{n-k}\,\mathbb{E}\,|\xi_k| \to 0,\quad n\to\infty.
\]
Hence, by the Chebyshev inequality, for any $\varepsilon>0$
\[
P(|x_n|\geqslant \varepsilon)\leqslant \frac{\mathbb{E}\,|x_n|}{\varepsilon}\,\to 0,\quad n\to\infty.
\]
\end{proof}
\begin{lemma}
    For an arbitrary $[a, b]\subset (0, 1)$ and any $\varepsilon>0$
    \[
    \lim\limits_{n\to\infty}\,\frac{1}{n^{1-a-\varepsilon}}\,\sum\limits_{k=1}^n\,\mathbb{I}_{[a, b]}(x_k) =  +\infty,\quad a.s.
    \]
    For any $\delta > 0$
    \[
    \lim\limits_{n\to\infty}\,\frac{1}{n}\,\sum\limits_{k=1}^n\,\mathbb{I}_{[-\delta, \delta]}(x_k)=1,\quad a.s.
    \]
\end{lemma}
\begin{proof}
    Denote
\[
\mu_n=P\circ x_n^{-1}, \quad \zeta_n = P\circ \left(\frac{\eta_{n}}{\sqrt{2\ln (n+1)}}\right)^{-1},\quad n\geqslant 1.
\]
In Theorem \ref{ThmVisitationMeasuresIterations} it was proved that the sequence $\{\varkappa_n = \sum\limits_{k=1}^n\,\mu_k,\, n\geqslant 1\}$ is with probability one a sequence of visitation measures for $\{x_n,\,n\geqslant 1\}$. Therefore, it is enough to check that
\begin{equation}\label{Check1}
\lim\limits_{n\to\infty}\,\frac{1}{n^{1-a-\varepsilon}}\,\mu_n([a, b]) = +\infty,
\end{equation}
\begin{equation}\label{Check2}
\lim\limits_{n\to\infty}\,\frac{1}{n}\,\mu_n([-\delta, \delta]) = 1.
\end{equation}
Note that the measures $\mu_n$, $\zeta_n$ satisfy the following relation
\[
\mu_{n+1}=(\mu_n\circ \varphi^{-1})\star \zeta_{n+1},\quad n\geqslant 0
\] 
and therefore $\mu_n$ has the density
\[
q_n(u) = \int\limits_\mathbb{R}\,p_n(u-v)\,\mu_{n-1}\circ \varphi^{-1}(dv),
\]
where $p_n$ is the density of $\zeta_n$.

From Lemma \ref{LemmaPconvergenceIterations} it follows that for $\varepsilon_1\in \bigl(0, min((b-a)/2, \varepsilon)\bigr)$
\[
\exists n_0\geqslant 1\;\forall n\geqslant n_0\;\forall u\in [a+\varepsilon_1, a+2\varepsilon_1]\colon\qquad q_n(u)\geqslant \frac{1}{2}\,p_n(a+2\varepsilon_1)
\]
and this implies (\ref{Check1}). Condition (\ref{Check2}) is proven in a similar manner.
\end{proof}

\begin{corollary}
  For the sequence $\{x_n,\,n\geqslant 1\}$  of iterations given by (\ref{EqGaussIterations}) 
  \[
  S-\lim\limits_{n\to\infty}\,x_n = 0,\quad a.s.
  \]
\end{corollary}
Note that the example considered  in this section illustrates that visitation measures contain all the information about all the partial limits of the sequence, not just its statistical limit.

{\footnotesize

\vskip10pt

Institute of Mathematics of the National Academy of Sciences of Ukraine

Kyiv, Ukraine

Email andrey.dorogovtsev@gmail.com % Author's institution address, e-mail address.

National Technical University of Ukraine <<Igor Sikorsky Kyiv Polytechnic Institute>>

Kyiv, Ukraine

Email nishchenkoii-ipt@lll.kpi.ua 
\vskip10pt
}
\hfill {\it Received \ 14.06.2026


\begin{thebibliography}{9}
\parskip-2pt
\bibitem{Dorogovtsev1} A. A. Dorogovtsev \emph{Some characteristics of sequences of iterations with random perturbations}, Ukr. Math. J. \textbf{48}, (1996), 1182 -- 1201.  https://doi.org/10.1007/BF02383865.
\bibitem{Dorogovtsev2} A. A. Dorogovtsev \emph{Visiting measures and an ergodic theorem for a sequence of iterations with random perturbations}, Ukr Math J \textbf{51}, (1999), 134 -- 139. https://doi.org/10.1007/BF02591922
\bibitem{Fast} H. Fast \emph{Sur la convergence statistique}, Colloq.Math.\textbf{2}, (1951), 241 -- 244. 
\bibitem{Steinhaus} H. Steinhaus \emph{Sur la convergence ordinaire et la convergence asymptotique}, Colloq. Math. \textbf{2}, (1951), 73-74. 
\bibitem{Vlasenko} M. A. Vlasenko \emph{Local visitation measures for some sequences of random variables with de\-crea\-sing coefficients,} Theory of probability and Its Applications, \textbf{49} (1), (2005), 176 -- 186. 
doi. 10.1137/S0040585X97980919,
https://api.semanticscholar.org/CorpusID:123288655

\end{thebibliography}
\end{document}